\documentclass[12pt]{amsart}
\usepackage{amsfonts,amssymb,amscd,amsmath,enumerate,url,verbatim}
 \usepackage[dvips]{epsfig}
 \usepackage{amsmath,amsthm,amscd,amsfonts,amssymb,enumerate}
\usepackage{graphicx}
\usepackage{color}
\usepackage[colorlinks]{hyperref}

 \usepackage{amsgen, amstext,amsbsy,amsopn, amsthm, amsfonts,amssymb,amscd,amsmat
 h,euscript,enumerate,url,verbatim,calc,xypic}
\numberwithin{equation}{section}
\newtheorem{theorem}{Theorem}[section]
\newtheorem{thm}[theorem]{Theorem}
\newtheorem{lem}[theorem]{Lemma}
\newtheorem{prop}[theorem]{Proposition}
\newtheorem{cor}[theorem]{Corollary}
\newtheorem{defn}[theorem]{Definition}
\newtheorem{rem}[theorem]{Remark}

\newcommand{\Ass}{\mbox{Ass}}

\renewcommand{\dim}{\mbox{dim}\,}

\newcommand{\J}{\mbox{J}}

\newcommand{\fa}{\mathfrak{a}}
\newcommand{\fb}{\mathfrak{b}}

\newcommand{\fm}{\mathfrak{m}}

\begin{document}

\bibliographystyle{amsplain}

\title{On graded local cohomology modules defined by a pair of ideals }

\author[M. Jahangiri]{M. Jahangiri}
\address{M. Jahangiri\\Faculty of Mathematical Sciences and Computer,  Kharazmi
University,  Tehran, Iran AND  Institute for Research in Fundamental Sciences (IPM)  P. O.  Box: 19395-5746, Tehran, Iran.}
\email{jahangiri@khu.ac.ir}
\author[Kh. Ahmadi Amoli]{Kh. Ahmadi Amoli}
\address{ Kh. Ahmadi Amoli\\Payame Noor University, Po Box 19395-3697, Tehran, Iran}
\email{khahmadi@pnu.ac.ir}
\author[Z. Habibi]{Z. Habibi}
\address{Z. Habibi \\Payame Noor University, Po Box 19395-3697, Tehran, Iran}
\email{z\_habibi@pnu.ac.ir}



\thanks{The first author was in part supported by a grant from IPM (No.
93130111)}
 
\maketitle
\begin{abstract}

Let $R = \bigoplus_{n \in \mathbb{N}_{0}} R_{n}$ be a standard
graded ring, $M$ be a finite graded $R$-module and $J$ be a
homogenous ideal of $R$. In this paper we study the graded structure
of the $i$-th local cohomology  module of $M$ defined by a pair of
ideals $(R_{+},J)$, i.e. $H^{i}_{R_{+},J}(M)$. More precisely, we
discuss finiteness property and vanishing of the graded components
$H^{i}_{R_{+},J}(M)_{n}$.

Also, we study the Artinian property and tameness of certain
submodules and quotient modules of $H^{i}_{R_{+},J}(M)$.

\end{abstract}


\section{introduction}

Let $R$ denotes a commutative Noetherian ring, $M$ an $R$-module and
$I$ and $J$ stand for two ideals of $R$. Takahashi, et. all in
\cite{TAK} introduced the $i$-th local cohomology functor with
respect to $(I,J)$, denoted by $H^{i}_{I,J}(-)$, as the $i$-th right
derived functor of the $(I,J)$- torsion functor $\Gamma _{I,J}(-)$,
where $\Gamma _{I,J}(M):=\{x \in M : I^{n}x\subseteq Jx$ for $n\gg
1\}$. This notion is the ordinary local cohomology functor when
$J=0$ (see \cite{B-SH}). The main motivation for this generalization
comes from the study of a dual of ordinary local cohomology modules
$H^{i}_{I}(M)$ (see \cite{sch}). Basic facts and more information
about local cohomology defined by a pair of ideals can be obtained
from \cite{TAK}, \cite{chu1} and \cite{CH-W}.

 Now, let
$R=\oplus_{n\in \mathbb{N}_{0}}R_{n}$ be a standard graded
Noetherian ring, i.e. the base ring $R_{0}$ is a commutative
Noetherian ring and $R$ is generated, as an $R_{0}$-algebra, by
finitely many elements of $R_{1}$. Also, let $J$ be a homogenous
ideal of $R$, $M$ be a graded $R$-module and $R_{+}:=\oplus_{n\in
\mathbb{N}}R_{n}$ be the irrelevant ideal of $R$. It is well known
(\cite[Section 12]{B-SH}) that for all $i\geq0$ the $i$-th local
cohomology module $H^{i}_{J}(M)$ of $M$ with respect to $J$ has a
natural grading and that, in the case where $M$ is finite,
$H^{i}_{R_{+}}(M)_{n}$ is a finite $R_{0}$-module for all $n\in
\mathbb{Z}$ and vanishes for all $n\gg0$ (\cite[Theorem
15.1.5]{B-SH}).

 In this paper, first, we show that $H^{i}_{I,J}(M)$
has a natural grading, when $I$ and $J$ are homogenous ideals of $R$
and $M$ is a graded $R$-module. Then, we show that, although  in
spite of the ordinary case, $H^{i}_{R_{+},J}(M)_{n}$ might be
non-finite over $R_{0}$ for some $n \in \mathbb{Z}$ and non-zero for
all $n\gg0$, but in some special cases they are finite for all $n
\in \mathbb{Z}$ and vanishes for all $n\gg0$. More precisely, we
show that if $(R_{0},\fm_{0})$ is local,  $ R_{+}\subseteq \fb$ is
an ideal of $R$ and $\bigcap \limits _{k=0}^\infty
\fm_{0}^{k}H^{i}_{\fb,\fm_{0}R}(M)_{n}=0$ for all $n\gg0$, then
$H^{i}_{\fb,\fm_{0}R}(M)_{n}=0$ for all $n\gg 0$. Also, we present
an equivalent condition for the finiteness of  components
$H^{i}_{R_{+},J}(M)_{n}$ (Theorem \ref{m0}).

In the last section, first, we study the asymptotic stability of the
set $\{Ass_{R_{0}}(H^{i}_{R_{+},J}(M)_{n})\}_{n\in \mathbb{Z}}$ for
$n\rightarrow -\infty$  in a special case (Theorem \ref{ass}). Then
we present some results about Artinianness of some quotients of
$H^{i}_{R_{+},J}(M)$. In particular, we show that  if $R_{0}$ is a
local ring with maximal ideal $\fm_{0}$ and $c\in \mathbb{Z}$ such
that $H^{i}_{R_{+},\fm_{0}R}(M)$ is Artinian for all $i>c$, then the
$R$-module
$H^{c}_{R_{+},\fm_{0}R}(M)/\fm_{0}H^{c}_{R_{+},\fm_{0}R}(M)$ is
Artinian  (Theorem \ref{ch-w}). Finally, we show that
$H^{i}_{R_{+},J}(M)$ is "tame" in a special case (Corollary
\ref{tam}).

\section{graded local cohomology modules defined by a pair of ideals}

Let $R= \oplus_{n\in \mathbb{Z} }R_{n}$ be a graded ring, $I$ and
$J$ be two homogenous ideals of $R$ and $M= \oplus_{n\in \mathbb{Z}
}M_{n}$ be a graded $R$-module. Then, it is natural to ask whether
the local cohomology modules $H^{i}_{I,J}(M)$ for all ${i\in
\mathbb{N}_{0}}$, also carry structures as graded $R$-modules. In
this section we show that it has affirmative answer.

First we show that, the $(I,J)$-torsion functor $\Gamma_{I,J}(-)$
can be viewed as a (left exact, additive) functor from
$^\ast~\mathcal{C}$ to itself. Since the category
$^\ast~\mathcal{C}$, of all graded $R$-modules and homogeneous
$R$-homomorphisms, is an Abelian category which has enough
projective object and enough injective objects, we can therefore
carry out standard techniques of homological algebra in this
category. Hence, we can form the right derived functors
$^\ast~H^{i}_{I,J}(-)$ of $\Gamma_{I,J}(-)$ on the category
$^\ast~\mathcal{C}$.
\begin{lem}\label{ann}

 Let $x=x_{i_{1}}+\cdots+x_{i_{k}}\in M$ be such that
$x_{i_{j}}\in M_{i_{j}}$ for all $j=1,\cdots,k$. Then

\[r(Ann(x)) = \mathop{\cap}_{j=1}^k r(Ann (x_{i_{j}})) \]
\end{lem}

\begin{proof}
$\supseteq$: Is clear.

$\subseteq$: First we show that if $y=y_{l_{1}}+\cdots+y_{l_{m}}\in
Ann(x)$ such that $y_{l_{k}}\in R_{l_{k}}$ for all $k=1,\cdots,m$
and $l_{1}<l_{2}<\cdots<l_{m}$, then $y_{l_{1}}\in
\mathop{\cap}_{j=1}^k r(Ann(x_{i_{j}}))$. We have
$$0=yx=\sum_{j=1}^{n} \sum_{k=1}^{m}
y_{l_{k}}x_{i_{j}},\hspace{6cm}(*)$$ comparing degrees, we get
$y_{l_{1}}x_{i_{1}}=0$. Let $j_{0}>1$ and suppose, inductively, that
for all $j'<j_{0}$, $y_{l_{1}}^{j'}x_{i_{j'}}=0$. Then using $(*)$
we get

$$\sum_{j=1}^{n} \sum_{k=1}^{m}
y_{l_{k}}y_{l_{1}}^{j_{0}-1}x_{i_{j}}=0.$$ Again, comparing degrees,
we have $y_{l_{1}}^{j_{0}}x_{i_{j_{0}}}=0$. So, $y_{l_{1}}\in
 \mathop{\cap}_{j=1}^k r(Ann (x_{i_{j}})) $.

Now, let $y=y_{l_{1}}+\cdots+y_{l_{m}}\in r(Ann(x))$ such that
 $l_{1}<l_{2}<\ldots <l_{m}$ and $y_{l_{j}}\in R_{l_{j}}$ for all $j=1,\cdots, m$. Then, there exists $s\in \mathbb{N}_{0}$ such
that $y^{s}x=0$.

 In order to show that $y\in \mathop{\cap}_{j=1}^k
r(Ann (x_{i_{j}}))$ we proceed by induction on $m$.
 The result is
clear in the case $m=1$. Now, suppose inductively that $m>1$ and the
result has been proved for values less than $m$. Using the above
agreement, we know that $y_{l_{1}}\in
\overset{k}{\underset{j=1}{\cap}}r(Ann (x_{i_{j}}))\subseteq
r(Ann(x))$. Then $y_{l_{2}}+\cdots+y_{l_{m}}=y-y_{l_{1}}\in
r(Ann(x))$. By inductive hypothesis, $y_{l_{2}}+\cdots+y_{l_{m}}\in
\overset{k}{\underset{j=1}{\cap}}r(Ann (x_{i_{j}}))$ and so, $y\in
\overset{k}{\underset{j=1}{\cap}}r(Ann (x_{i_{j}}))$.

\end{proof}
\begin{lem}
 $\Gamma _{I,J}(M)$ is a graded $R$-module.
\end{lem}
\begin{proof}
Let $x\in \Gamma _{I,J}(M)$. Assume that
$x=x_{i_{1}}+\cdots+x_{i_{k}}$ where for all $j=1,2,\cdots,k$,
$x_{i_{j}}\in M_{i_{j}}$ and $i_{1}<i_{2}<\cdots<i_{k}$. We show
that $x_{i_{1}},\cdots,x_{i_{k}}\in \Gamma _{I,J}(M)$. Since $R$ is
Noetherian, there is $t'\in \mathbb{N}$ such that $(r(Ann
x_{i_{j}}))^{t'}\subseteq Ann(x_{i_{j}})$ for all $j=1,2,\cdots,k$.
Let $n\in \mathbb{N}_{0}$ be such that $I^{n}\subseteq Ann(x)+J$.
So, by Lemma \ref{ann}, for all $j=1,2,\cdots,k$ we have
$$I^{2nt'}\subseteq
(Ann(x)+J)^{2t'}\subseteq(Ann(x))^{t'}+J^{t'}\subseteq
(r(Ann(x)))^{t'}+J^{t'}=({\cap}_{j=1}^k r(Ann
(x_{i_{j}})))^{t'}+J^{t'}$$$$\subseteq Ann(x_{i_{j}})+J.$$

Thus $x_{i_{j}}\in \Gamma _{I,J}(M)$, as required.
\end{proof}
To calculate graded local cohomology module $^\ast~H^{i}_{I,J}(M)$
(${i\in \mathbb{N}_{0}}$), one proceeds as follows:

Taking an $^\ast$injective resolution

$$E^{\bullet}:  0\rightarrow
E^{0}\xrightarrow {d^{0}}E^{1}\xrightarrow {d^{1}}E^{2}\rightarrow
\cdots\rightarrow E^{i}\xrightarrow {d^{i}} E^{i+1}\rightarrow \cdots,$$

 of
$M$ in $^\ast~\mathcal{C}$, applying the functor $\Gamma _{I,J}(-)$
to it and  taking the i-th cohomology module of this complex, we get
$$\frac{ker\Gamma _{I,J}(d^{i}) }{im\Gamma
_{I,J}(d^{i-1})}$$ which is denoted by $^\ast~{H^{i}_{I,J}(M)}$ and
is a graded $R$-module.


\begin{rem}

Let $0 \rightarrow L\xrightarrow{f}M\xrightarrow{g} N\rightarrow 0$
be an exact sequence of $R$-modules and $R$-homomorphisms. Then, for
each $i \in \mathbb{N}_{0}$, there is a homogeneous connecting
homomorphism

$^\ast~{H^{i}_{I,J}(N)}\rightarrow^\ast~{H^{i+1}_{I,J}(L)}$ and
these connecting homomorphisms make the resulting homogenous long
exact sequence

 $$\begin{array}{lll}0\rightarrow
^\ast~H^{0}_{I,J}(L)&\xrightarrow{^\ast~H^{0}_{I,J}(f)} ^\ast~H^{0}_{I,J}(M)&\xrightarrow
{^\ast~H^{0}_{I,J}(g)}^\ast~H^{0}_{I,J}(N)
\\\rightarrow^\ast~H^{1}_{I,J}(L)&\xrightarrow{^\ast~H^{1}_{I,J}(f)} ^\ast~H^{1}_{I,J}(M)&\xrightarrow{^\ast~H^{1}_{I,J}(g)}^\ast~H^{1}_{I,J}(N)
\\\rightarrow\cdots
\\\rightarrow^\ast~H^{i}_{I,J}(L)&\xrightarrow{^\ast~H^{i}_{I,J}(f)}^\ast~H^{i}_{I,J}(M)&\xrightarrow{^\ast~H^{i}_{I,J}(g)}^\ast~H^{i}_{I,J}(N)
\\\rightarrow^\ast~H^{i+1}_{I,J}(L)\rightarrow \cdots
.\end{array}$$

The reader should also be aware of the 'natural' or 'functorial'
properties of these long exact sequences.

\end{rem}

\begin{defn}\label{sg}
We define a partial order on the set
\[^\ast~\widetilde{W}(I,J):=\{\fa: \fa \mbox{  is a homogenous ideal of } R; I^{n}\subseteq
\fa+J \mbox{ for some integer } n\geq 1\},\]

 by letting
$\fa\leq \fb$ if $\fa\supseteq \fb$, for $\fa,\fb\in
^\ast~\widetilde{W}(I,J)$. If $\fa\leq \fb$, then we have $\Gamma
_{\fa}(M)\subseteq \Gamma _{\fb}(M)$. Therefore, the relation $\leq$
on $^\ast~\widetilde{W}(I,J)$ with together the inclusion maps make
$\{\Gamma _{\fa}(M)\}_{\fa\in ^\ast~\widetilde{W}(I,J)}$ into a
direct system of graded R-modules.
\end{defn}

As Takahashi et. all in \cite{TAK} showed the relation between the
local cohomology functor $H^{i}_{I}(-)$ and $H^{i}_{I,J}(-)$, we
show the same relation between graded version of them as follows.

\begin{prop}\label{lim}

Let $M$ be a graded $R$-module. Then there is a natural graded
isomorphism

$$\big(^\ast~H^{i}_{I,J}(M)\big)_{i\in \mathbb{N}_0}\cong \big(\underset{\fa\in
^\ast~\widetilde{W}(I,J)}{\varinjlim}^\ast~H^{i}_{\fa}(M)\big)_{i\in
\mathbb{N}_0},$$ of strongly connected sequences of covariant
functors.

\end{prop}
\begin{proof}
First of all, we show that $\Gamma _{I,J}(M)=\underset{\fa\in
^\ast~\widetilde{W}(I,J)}{\bigcup \Gamma _{\fa}(M)}$.

$\supseteq$: Suppose that $x\in \underset{\fa\in
^\ast~\widetilde{W}(I,J)}{\bigcup \Gamma _{\fa}(M)}$. Then there are
$\fa\in ^\ast~\widetilde{W}(I,J)$ and integer $n$ such that
$I^{n}\subseteq \fa+J$ and $x\in \Gamma _{\fa}(M)$. Let $t\in
\mathbb{N}_{0}$ be such that $\fa^{t}\subseteq Ann(x)$. Therefore
$I^{2nt}\subseteq (\fa+J)^{2t}\subseteq \fa^{t}+J^{t}\subseteq
Ann(x)+J$ and so $x \in \Gamma _{I,J}(M)$.

$\subseteq$: Conversely, let $x\in \Gamma _{I,J}(M)$. Then
$I^{n}\subseteq Ann(x)+J$ for some $n\in \mathbb{N}$. We show that
$x\in \Gamma _{\fa}(M)$ such that $\fa=r(Ann(x))$. As $r(Ann(x))$ is
homogenous by Lemma \ref{ann}, and $I^{n}\subseteq Ann(x)+J$, we
have $r(Ann(x))\in ^\ast~\widetilde{W}(I,J)$ and $x\in \Gamma
_{r(Ann(x))}(M)$.

Now,   \cite[Exercise 12.1.7]{B-SH} implies the desired isomorpism.

 \end{proof}

\begin{rem}
If we forget the grading on $^\ast~H^{i}_{I,J}(M)$, the resulting
$R$-module is isomorphic to $H^{i}_{I,J}(M)$. More precisely, using
\cite[Proposition 12.1.3]{B-SH} and the fact that the direct systems
$\widetilde{W}(I,J)$ and $^\ast~\widetilde{W}(I,J)$ are cofinal we
have
$$H^{i}_{I,J}(E)\cong \underset{\fa\in
\widetilde{W}(I,J)}{\varinjlim}H^{i}_{\fa}(E)\cong \underset{\fa\in
^\ast~\widetilde{W}(I,J)}{\varinjlim}H^{i}_{\fa}(E)=0,$$ for all
$i>0$ and all  $^\ast$injective $R$-module $E$. Now, using similar
argument as used in \cite[Corollary 12.3.3]{B-SH}, one can see that
there exists an equivalent of functors

$$H^{i}_{I,J}(-\rceil {_{^\ast\mathcal{C}}})\cong^\ast~H^{i}_{I,J}(-)$$

from $^\ast~\mathcal{C}$ to itself.
\end{rem}

As a consequence of the above remark and \cite[Remark
13.1.9(ii)]{B-SH}, we have the following.

\begin{cor}
Let $t\in \mathbb{Z}$,  then
$$ H^{i}_{I,J}(M(t))\cong (H^{i}_{I,J}(M))(t),$$
for all $i\in \mathbb{N}$,  where $(.)(t):^\ast~\mathcal{C}\rightarrow
^\ast~\mathcal{C}$ is the t-th shift functor.
\end{cor}

\section{vanishing and finiteness of components}

A crucial role in the study of the graded local cohomology is
vanishing and finiteness of their components. As one can see in
Theorem 15.1.5 \cite{B-SH}, $H^{i}_{R_{+}}(M)_{n}$ is a finite
$R_{0}$-module for all $n\in \mathbb{Z}$ and it vanishes for all
$n\gg0$. In this section we show that, although it is not the same
for $H^{i}_{R_{+},J}(M)$, but it holds in some special cases.

In the rest of this paper, we assume that $R=\bigoplus_{n \in
\mathbb{N}_{0}} R_{n}$ is a standard graded ring and $M$ is a finite
graded $R$-module.

 Local cohomology with respect to a pair of ideals
does  not satisfy in Theorem 15.1.5 \cite{B-SH}, in general, as the
following counterexample shows.

\begin{rem}

\noindent(i) Let $R=\mathbb{Z}[X]$ and $R _{+}=(X)$. We can see that
$\Gamma _{R _{+},R
 _{+}}(\mathbb{Z}[X])_{n}=\mathbb{Z}[X]_{n}\neq 0$   for all $n\in \mathbb{N}_{0}$.

(ii) Assume that $J$ is an ideal of $R$ generated by elements of
degree zero such that $JR_{+}=0$. It is easy to see that in this
condition $\Gamma _{R_{+},J}(M)=\Gamma _{R_{+}}(M)$ and therefore,
\cite[Theorem 15.1.5]{B-SH} holds for $H^{i}_{R_{+},J}(M)$.

(iii) Let $(R_{0},\fm_{0})$ be a local ring and $\dim R_{0}=0$. Then
$\Gamma_{R_{+},\fm_{0}R}(M)=\Gamma_{R_{+}}(M)$ and, again,
\cite[Theorem 15.1.5]{B-SH} holds for $H^{i}_{R_{+},\fm_{0}R}(M)$.

\end{rem}

The following proposition, indicates a vanishing property on the
graded components of $ H^{i}_{\fb,\fm_{0}R}(M)$ for ideal
$\fb=\fb_{0}+R_{+}$ where $\fb_{0}$ is an ideal of $R_{0}$ and
$\fm_{0}$ is the unique maximal ideal of $R_0$. Vanishing of the
components $ H^{i}_{\fb }(M)_n$ for $n\gg 0$ has already been
studied in \cite{JZ}.

\begin{thm}\label{m01}

Assume that $(R_{0},\fm_{0})$ is local and $i\in \mathbb{N}_0$. Let
$\fb:=\fb_{0}+R_{+}$ where $\fb_{0}$ is an ideal of $R_{0}$ such
that for all finite graded $R$-module $M$, $\bigcap \limits
_{k=0}^\infty \fm_{0}^{k}H^{i}_{\fb,\fm_{0}R}(M)_{n}=0$ for all
$n\gg0$. Then $ H^{i}_{\fb,\fm_{0}R}(M)_{n}=0 $ for all $n\gg 0$ and
 all finite graded $R$-module $M$.

\end{thm}

\begin{proof}

We proceed by induction on $\dim M$.

 Let $J:=\fm_{0}R$. If $\dim
M=0$, then using \cite[Theorem 1]{KR} $\Gamma_{\fb,J}(M)_{n}= M_n=
0$ for all $n\gg 0$.

 Now, let $\dim M>0$. Considering the long exact sequence
$$H^{i}_{\fb,J}(\Gamma_{J}(M))_{n}\rightarrow H^{i}_{\fb,J}(M)_{n}\rightarrow
H^{i}_{\fb,J}(\overline{M})_{n}\rightarrow
H^{i+1}_{\fb,J}(\Gamma_{J}(M))_{n},$$ where
$\overline{M}=M/\Gamma_{J}M$, by \cite[Proposition 1.1]{JZ} we get
$H^{i}_{\fb,J}(M)_{n}\cong H^{i}_{\fb,J}(\overline{M})_{n}$ for all
  $n\gg 0$. Therefore, we may assume
that $M$ is $J$-torsion free and so   there exists
$x_{0}\in\fm_{0}\backslash Z_{R_{0}}(M)$. Now, the exact sequence
$$0\rightarrow M\xrightarrow{x_{0}} M\rightarrow M/x_{0
}M\rightarrow 0$$ implies the   exact sequence
$$
H^{i}_{\fb,J}(M)_{n}\xrightarrow{x_{0}}
H^{i}_{\fb,J}(M)_{n}\rightarrow H^{i}_{\fb,J}(M/x_{0}M)_{n}.$$

Then, by the assumptions and the inductive hypothesis,
$$H^{i}_{\fb,J}(M/x_{0}M)_{n}=0$$ for all   $n\gg0$.
So,
$$H^{i}_{\fb,J}(M)_{n}= x_{0}H^{i}_{\fb,J}(M)_{n}$$ for all
$n\gg0$. Therefore,
$$H^{i}_{\fb,J}(M)_{n}= \bigcap \limits _{k=0}^\infty
x_{0}^{k}H^{i}_{\fb,J}(M)_{n}= 0$$ for all $n\gg0$. Now, the result
follows by induction.

\end{proof}

In the following we present an equivalent condition for the
finiteness of components $H^{i}_{R _{+},J}(M)_{n}$.

\begin{thm}\label{m0}

Let $(R_{0},\fm_{0})$ be   local  and $J_{0}\subseteq\fm_{0}$  be an
ideal of $R_0$. Then the following statements are equivalent.
\begin{enumerate}
\item[a)] For all finite graded $R$-module  $M$ and all  $i \in \mathbb{N}_{0}$, $H^{i}_{R _{+},J_{0}R}(M)_{n} = 0$ for  $n\gg0$.
\item[b)] For all finite graded $R$-module  $M$, all $i \in \mathbb{N}_{0}$ and $n\in\mathbb{Z}$, $H^{i}_{R _{+},J_{0}R}(M)_{n}$ is a finite $R_{0}$-module.
\end{enumerate}
\end{thm}

\begin{proof}
Let $J=J_{0}R$.

  a)$\Rightarrow$ b) Let $M$ be a non-zero finite graded
$R$-module. We proceed by induction on $i$. It is clear that
$H^{0}_{R_{+},J}(M)$ is a finite $R$-module and then
$H^{0}_{R_{+},J}(M)_{n}$ is finite as an $R_{0}$-module for all
$n\in \mathbb{Z}$.

Now, suppose that $i>0$ and  the result is proved for smaller values
than $i$. As $H^{i}_{R_{+},J}(M)\cong H^{i}_{R_{+},J}(M/\Gamma
_{R_{+},J}(M))$, we may assume that $M$ is an $(R_{+},J)$-torsion
free $R$-module and so $R_{+}$-torsion free $R$-module. Hence
$R_{+}$ contains a non zero-divisor on $M$. As $M\neq R_{+}M$, there
exists a homogeneous element $x\in R_{+}$ of degree $t$, which is a
non zero-divisor on $M$,  by \cite[Lemma 15.1.4]{B-SH}. We use the
exact sequence $0\rightarrow M \xrightarrow {x}
M(t)\rightarrow(M/xM)(t)\rightarrow 0$ of graded $R$-modules and
homogeneous homomorphisms to obtain the exact sequence $$
H^{i-1}_{R_{+},J}(M/xM)_{n+t}\rightarrow
H^{i}_{R_{+},J}(M)_{n}\xrightarrow {x} H^{i}_{R_{+},J}(M)_{n+t}$$
for all $n\in \mathbb{Z}$. It follows from the inductive hypothesis
that $H^{i-1}_{R_{+},J}(M/xM)_{j}$ is a finite $R_{0}$-module for
all $j \in \mathbb{Z}$. Let $s\in \mathbb{Z}$ be such that
$H^{i}_{R_{+},J}(M)_{m}=0$ for all $m\geq s$. Fix an integer $n$,
then for some $k\in \mathbb{N}_{0}$ we get $n + kt \geq s$ and then
$H^{i}_{R_{+},J}(M)_{n+kt}=0$. Now, for all $j = 0,\cdots,k-1$, we
have the exact sequence
$$H^{i-1}_{R_{+},J}(M/xM)_{n+(j+1)t}\rightarrow
H^{i}_{R_{+},J}(M)_{n+jt}\xrightarrow {x}
H^{i}_{R_{+},J}(M)_{n+(j+1)t}.$$

Since $H^{i}_{R_{+},J}(M)_{n+kt}=0$ and
$H^{i-1}_{R_{+},J}(M/xM)_{n+kt}$ is a finite $R_{0}$-module, so
 $H^{i}_{R_{+},J}(M)_{n+(k-1)t}$ is a finite $R_{0}$-module.
Therefore $H^{i}_{R_{+},J}(M)_{n+jt}$ is a finite $R_{0}$-module for
$j= 0,\ldots, k- 1$. Now, the result follows by induction.

 b)$\Rightarrow$ a)  The result follows from the above theorem.

\end{proof}

\section{Asymptotic behavior of $H^{i}_{R_{+},J}(M)_{n}$ for $n\ll 0$}
In this section  we consider the asymptotic behavior of components
$H^{i}_{R_{+},J}(M)_{n}$ when $n\rightarrow -\infty$. More
precisely, first we study the asymptotic  stablity  of the set
\\$\{Ass_{R_{0}}(H^{i}_{R_{+},J}(M)_{n})\}_{n\in \mathbb{Z}}$ in a special case.
Then, we investigate the Artinianness and tameness of some quotients
and submodules of $H^{i}_{R_{+},J}(M)$.

Let us recall that for a given   sequence $\{S_{n}\}_{n\in\mathbb{Z}}$
of sets $S_{n}\subseteq Spec(R_{0})$, we say that  $\{S_{n}\}_{n\in \mathbb{Z}}$ is
asymptotically stable for $n\rightarrow -\infty$, if there is some
$n_0\in\mathbb{Z}$ such that $S_{n}=S_{n_{0}}$ for all $n\leq n_{0}$
(see \cite{B}). Let the base ring $R_{0}$ be local and $i \in
\mathbb{N}_{0}$ be such that the $R$-module $H^{j}_{R_{+}}(M)$ is
finite for all $j<i$. In \cite[Lemma 5.4]{B-H} it has been shown
that $\{\Ass_{R_{0}}(H^{i}_{R_{+}}(M)_{n})\}_{n\in \mathbb{Z}}$ is asymptotically stable
for $n\rightarrow -\infty$. The next theorem use similar argument to
improve this result to local cohomology modules defines by a pair of
ideals.

\begin{thm}\label{ass}
Let $(R_0,\fm_{0})$ be a local ring with infinite residue field and
$i\in\mathbb{N}_{0}$ be such that the $R$-module
$H^{j}_{R_{+},J}(M)$ is finite for all $j<i$.  If one of the
equivalent conditions of the Theorem \ref{m0} holds, then
$\Ass_{R_{0}}(H^{i}_{R_{+},J}(M)_{n})$ is asymptotically stable for
$n\rightarrow -\infty.$

\end{thm}
\begin{proof}
We use induction on $i$. For $i=0$ the result is clear from the fact
that $ H^{0}_{R _{+},J}(M)_{n}=0$ for all $n\ll 0$. Now, let $i>0$.
In view of the natural graded isomorphism, $H^{i}_{R_{+},J}(M)\cong
H^{i}_{R_{+},J}(M/\Gamma _{R_{+},J}(M))$, for all $i \in
\mathbb{N}_{0}$, and using \cite[Lemma 15.1.4]{B-SH}, we may assume
that  there exists a homogeneous element $x\in R_{1}$ which is a non
zero-divisor on $M$. Now, by the long exact sequence $$
H^{j-1}_{R_{+},J}(M)\rightarrow H^{j-1}_{R_{+},J}(M/xM)\rightarrow
H^{j}_{R_{+},J}(M)(-1)\xrightarrow {x} H^{j}_{R_{+},J}(M)$$ for all
$j\in \mathbb{Z}$, we have $H^{j }_{R_{+},J}(M/xM)$ is finite for
all $j< i- 1$. Hence, by the inductive hypothesis,
$$\Ass_{R_{0}}(H^{i-1}_{R_{+},J}(M/xM)_{n})=\Ass_{R_{0}}(H^{i-1}_{R_{+},J}(M/xM)_{n_{1}})=: X$$
for some $n_{1} \in \mathbb{Z}$ and all $n\leq n_{1}$. Furthermore,
there is some $n_{2}< n_{1}$ such that
$H^{i-1}_{R_{+},J}(M)_{n+1}=0$ for all $n\leq n_{2}$. Then for all
$n\leq n_{2}$ we have the exact sequence
$$0\rightarrow H^{i-1}_{R_{+},J}(M/xM)_{n+1}\rightarrow
H^{i}_{R_{+},J}(M)_{n}\xrightarrow {x} H^{i}_{R_{+},J}(M)_{n+1}.$$
Thus, it shows that $$X\subseteq
\Ass_{R_{0}}(H^{i}_{R_{+},J}(M)_{n})\subseteq X\cup
\Ass_{R_{0}}(H^{i}_{R_{+},J}(M)_{n+1})$$ for all $n\leq n_{2}$.
Hence $$\Ass_{R_{0}}(H^{i}_{R_{+},J}(M)_{n})\subseteq
\Ass_{R_{0}}(H^{i}_{R_{+},J}(M)_{n+1})$$ for all $n< n_{2}$ and,
using the assumption, the proof is complete.

\end{proof}

In the rest of paper, we pay attention to the Artinianness property
of the graded modules $H^{i}_{R_{+},J}(M)$. The following
proposition, gives a graded analogue of \cite[Theorem 2.2]{CH-W}.

\begin{thm}\label{ch-w}
Assume that $R_{0}$ is a local ring with maximal ideal $\fm_{0}$. If
$c\in \mathbb{Z}$ and $H^{i}_{R_{+},\fm_{0}R}(M)$ is Artinian for
all $i>c$, then the $R$-module
$H^{c}_{R_{+},\fm_{0}R}(M)/\fm_{0}H^{c}_{R_{+},\fm_{0}R}(M)$ is
Artinian.

\end{thm}

\begin{proof}
Let $\fm:=\fm_{0}+R_{+} $ be the unique graded maximal ideal of $R$
and let $J:=\fm_{0}R$. We have $H^{i}_{R_{+},J}(M)=H^{i}_{\fm,J}(M)$
for all $i$. Thus we can replace $R_{+}$ by $\fm$. We proceed the
assertion by induction on $n:=\dim M$. The result is clear in the
case $n=0$. Let $n>0$ and that the statement is proved for all
values less than $n$. Now, using the long exact sequence

$$H^{i}_{\fm,J}(\Gamma _{\J}(M)\rightarrow H^{i}_{\fm,J}(M)\rightarrow H^{i}_{\fm,J}(M/\Gamma
_{J}(M))\rightarrow H^{i+1}_{\fm,J}(\Gamma _{\J}(M)), $$   and the
fact that  $H^{i}_{\fm,J}(\Gamma _{J}(M))=H^{i}_{\fm}(\Gamma
_{J}(M))$ is Artinian for all $i$, replacing $M$ with $M/\Gamma
_{J}(M)$, we may assume that $\Gamma _{J}(M)=0$. Therefore,  there
exists $x_{0}\in\fm_{0}\backslash Z_{R_{0}}(M)$. Now, the  long
exact sequence
$$H^{i}_{\fm,J}(M)\xrightarrow{x_{0}}
H^{i}_{\fm,J}(M)\xrightarrow{\alpha_i}
H^{i}_{\fm,J}(M/x_{0}M)\xrightarrow{\beta_i}H^{i+1}_{\fm,J}(M)$$
implies that $H^{i}_{\fm,J}(M/x_{0}M)$ is Artinian for all $i>c$ and
so, by inductive hypothesis,
$H^{c}_{\fm,J}(M/x_{0}M)/\fm_{0}H^{c}_{\fm,J}(M/x_{0}M)$ is
Artinian. Considering the   exact sequences

$$0\rightarrow Im\alpha_c\rightarrow
H^{c}_{\fm,J}(M/x_{0}M)\rightarrow Im\beta_c\rightarrow 0$$ and
$$H^{c}_{\fm,J}(M)\xrightarrow{x_{0}}H^{c}_{\fm,J}(M)\xrightarrow{\alpha_c} Im\alpha_c \rightarrow
0,$$ we get the following exact sequences
 $$\begin{array}{lll}Tor_{1}^R(R_{0}/\fm_{0}, Im\beta_c)&\rightarrow Im\alpha_c/\fm_{0}Im\alpha\rightarrow& H^{c}_{\fm,J}(M/x_{0}M)/\fm_{0}H^{c}_
 {\fm,J}(M/x_{0}M)  \end{array}\hspace{0.4cm}(A)$$
and
$$H^{c}_{\fm,J}(M)/\fm_{0}H^{c}_{\fm,J}(M)\xrightarrow{x_{0}} H^{c}_{\fm,J}(M)/\fm_{0}H^{c}_{\fm,J}(M)\rightarrow Im\alpha_c/\fm_{0}Im\alpha_c\rightarrow
0.\hspace{1.4cm}(B)$$ these two exact sequences implies that
$H^{c}_{\fm,J}(M)/\fm_{0}H^{c}_{\fm,J}(M)$ is Artinian and the
assertion follows.
\end{proof}

Let $I,J$ be ideals of $R$. Chu and Wang in \cite{CH-W} defined
$cd(I,J,R):= \sup\{i ; H^{i}_{I,J}(M)\neq0\}$.

The following corollary is an immediate consequence of Theorem
\ref{ch-w}.

\begin{cor}\label{cd}
Assume that $R_{0}$ is  local with maximal ideal $\fm_{0}$. If
$c:=cd(R_{+},\fm_{0}R,M)$. Then
$H^{c}_{R_{+},\fm_{0}R}(M)/\fm_{0}H^{c}_{R_{+},\fm_{0}R}(M)$ is
Artinian.
\end{cor}
Let $T=\bigoplus_{n \in \mathbb{N}_{0}} T_{n}$ be a graded
$R$-module. Following \cite{B}, we say that $T$ is tame or
asymptotically gap free if either $T_{n}=0 $ for all $n\ll 0$ or
else $T_{n}\neq 0 $ for all $n\ll 0$. Now, as an application of the
above  Corollary, we have the following.
\begin{cor}\label{tam}
Let $(R_0, \fm_0)$ be local and $c := cd(R_{+},\fm_{0}R,M)$. If one
of the equivalent conditions of Theorem \ref{m0} holds, then
$H^{c}_{R_{+},\fm_{0}R}(M)$ is tame.
\end{cor}
\begin{proof}
Let $J=\fm_{0}R$. Since $H^{c}_{R_{+},
\fm_{0}R}(M)/\fm_{0}H^{c}_{R_{+}, \fm_{0}R}(M)$ is Artinian so it is
tame. Now, the result follows using Nakayama's lemma.
\end{proof}

\begin{prop}\label{hj}
Let $(R_{0}, \fm_{0})$ be   local, $J\subseteq R_{+}$ be a
homogenous ideal of $R$ and $g(M):= \sup \{i : \forall j < i,
\ell_{R_{0}}(H^{j}_{R_{+}}(M)_{n}) <\infty, \forall n\ll0 \}$ be
finite. Then, the graded $R$-module
$H^{i}_{\fm_{0}R,J}(H^{j}_{R_{+}}(M))$ is Artinian for $i = 0, 1$
and all $j\leq g(M)$.
\end{prop}
\begin{proof}
Since $J\subseteq R_{+}$, so $H^{j}_{R_{+}}(M)$ is $J$-torsion.
Therefore, $H^i_{\fm_{0}R,J}(H^{j}_{R_{+}}(M))\cong
H^i_{\fm_{0}R}(H^{j}_{R_{+}}(M))$. Now, the result follows from
\cite[Theorem 2.4]{HJZ}.
\end{proof}

\end{document}